\documentclass[12pt,a4paper]{amsart}

\usepackage[x11names]{xcolor}

\usepackage{amstext, amsmath, amssymb, amsfonts, amsthm, mathtools, latexsym}
\usepackage{graphicx,tikz}

\usepackage{enumitem}
\usepackage[left=2cm, right=2cm, bottom=3cm]{geometry} 
\usepackage[skip=0.4cm,indent=0.5cm]{parskip}
\usepackage{indentfirst}
\usepackage{multicol}
\usepackage{hyperref}

\newcommand{\R}{\ensuremath{\mathbb{R}}}
\newcommand{\N}{\ensuremath{\mathbb{N}}}

\newcommand{\Pol}{\ensuremath{\mathbb{P}}}
\newcommand{\Exp}{\ensuremath{\mathbb{E}}}

\newcommand{\OR}{\ensuremath{\mathcal{O}}}
\newcommand{\OG}{\ensuremath{\mathbb{O}}}
\newcommand{\COG}{\ensuremath{\overline{\mathbb{O}}}}
\renewcommand{\L}{\ensuremath{\mathbb{L}}}
\newcommand{\CL}{\ensuremath{\overline{\mathbb{L}}}}
\newcommand{\B}{\ensuremath{\mathbb{B}}}
\newcommand{\CB}{\ensuremath{\overline{\mathbb{B}}}}

\renewcommand{\H}{\ensuremath{\mathcal{H}}}

\newcommand{\U}{\ensuremath{\mathcal{U}}}
\newcommand{\V}{\ensuremath{\mathcal{V}}}
\newcommand{\F}{\ensuremath{\mathcal{F}}}
\newcommand{\A}{\ensuremath{\mathcal{A}}}

\renewcommand{\S}{\ensuremath{\mathcal{S}}}
\newcommand{\e}{\ensuremath{\varepsilon}}
\newcommand{\f}{\ensuremath{\varphi}}

\numberwithin{figure}{section}

\newtheorem{theorem}{Theorem}[section]

\newtheorem{lemma}[theorem]{Lemma}
\newtheorem{corollary}{Corollary}

\newtheorem{proposition}[theorem]{Proposition}

\title{Flexibility of generalized entropy for wandering dynamics}

\author{Javier Correa}
\address{Universidade Federal de Minas Gerais}
\email {jcorrea@mat.ufmg.br}

\author{Hellen de Paula} 
\address{Universidade Federal de Minas Gerais} 
\email {hellenlimadepaula@ufmg.br}

\begin{document}

\begin{abstract}
We show a flexibility result in the context of generalized entropy. The space of dynamical systems we work with are homeomorphisms on the sphere whose non-wandering set consists of only one fixed point.

\end{abstract}

 \maketitle

\vspace{2cm}
\\


\section{Introduction}

An important problem in dynamical systems is the measuring of the complexity of a map in terms of its orbits. The topological entropy of a system studies the exponential growth rate at which orbits are separated, and it has become a crucial tool for classifying highly chaotic dynamical systems. However, there are many interesting families of dynamical systems, where every system has a vanishing entropy, and therefore, another object is needed. The object that studies the polynomial growth rate is called polynomial entropy, and this was introduced by J. P. Marco in  \cite{Ma13}. Interestingly enough, ten years before, S. Galatolo introduced in \cite{Ga03} a generalization of topological entropy that extends classical and polynomial entropy for one-parameter families of orders of growth. This concept is a translation of A. Katok and J. P. Thouvenot's notion of slow entropy defined in \cite{KaTh97}.   For a more detailed exposition on the works related to the previous quantities, check  \cite{KaKaWe20}, \cite{CoPa23}  and \cite{CoPu21}. In this article, we are going to work with the notion of generalized entropy introduced by the first author and E. Pujals in \cite{CoPu21}. This object, instead of quantifying the complexity of a system with a single number, works directly in the space of the orders of growth $\mathbb{O}$. In subsection \ref{subOGGE}, we show how $o(f)$, the generalized entropy of a map $f$, is constructed.

In this text, we are interested in the families of dynamical systems on the sphere whose non-wandering set consists of only one fixed point. We shall use $S^2$ to denote the sphere and $\Omega(f)$ to denote the non-wandering set of a map $f:S^2\to S^2$. We define $\H$ as the family of homeomorphisms $f:S^2\to S^2$, such that $\#\Omega(f)=1$.  A simple consequence of the variational principle is that any map in $\H$ has vanishing classical entropy. The family  $\H$ is naturally interesting because the orientation preserving ones are in a direct bijection with Brouwer homeomorphisms of the plane (by compactifying the plane with one point).

For the family $\H$, in the context of polynomial entropy, we have the result of  L. Hauseux and F. Le Roux in \cite{HaRo19}. We shall call the polynomial entropy of a map $h_{pol}(f)$. 

\begin{theorem}[L. Haseux and F. Le Roux]\label{theoHLR}
The following happens:
\begin{enumerate}
\item For any $f\in \H$, $h_{pol}(f)\in \{1\}\cup [2,\infty]$. 
\item For any $t\in \{1\}\cup [2,\infty]$, there exists $f\in \H$ such that $h_{pol}(f)=t$. Moreover, $h_{pol}(f)=1$ if and only if $f$ is the compactification of a map conjugate to a translation in $\R^2$. 
\end{enumerate}
\end{theorem}

The first part of the theorem gives natural boundaries and the second part is a flexibility result in the sense introduced by J. Bochi, A. Katok and F. Rodriguez-Hertz in \cite{BoKaRH19}.  The article \cite{HaRo19} has two parts. First, a way to code the orbits of a dynamical system with  $\#\Omega (f)=1$, which allows the computation of $h_{pol}(f)$ and second, the construction of maps in the point $(2)$ of theorem \ref{theoHLR}. We have extended the first part in \cite{CoPa23} to generalized entropy and also to maps with a finite non-wandering set. With this tool, we computed the generalized entropy of Morse-Smale diffeomorphisms on the surfaces and some maps in the boundary of chaos. Our main result in this article is an extension of the second part of the work by L. Hauseux and F. Le Roux.

In order to enunciate our theorem, we need to fix some terminology associated to $\OG$ the space of orders of growth. To us, an order of growth is the class of non-decreasing sequences in $(0,\infty)$ which have the same asymptotic speed  at $\infty$. We represent these classes by $[a(n)]$ for the general case and the formula between brackets if it is defined by one, for instance $[n^2]$. The family of polynomial orders of growth is $\Pol=\{[n^t]:0<t<\infty\}$. We say that an order of growth $[a(n)]$ verifies the linearly invariant property (LIP) if for some $m\geq 2$, $[a(mn)]=[a(n)]$. The set of orders of growth that verifies LIP is represented by $\L$ (observe that $\Pol \subset \L$).  The abstract completion of $\OG$ by its partial order is the complete set of orders of growth $\COG$ and the generalized entropy of a map $o(f)$ is defined in $\COG$. General elements of $\COG$ are represented by $o$. We define $\CL$ as the countable completion of $\L$ in $\COG$ (Check subsections \ref{subOGGE}, \ref{subBJP} and \ref{subLIP} for more details).

 By our results in \cite{CoPa23}, it is immediate that for any $f\in \H$, either $o(f)=[n]$ or $[n^2]\leq o(f) \leq \sup(\Pol)$. The main result in this article is the following. 

\begin{theorem}\label{teoFlex}
For any $o\in \CL$ verifying $o=[n]$ or $[n^2]\leq o \leq \sup(\Pol)$, there exists $f\in \H$ such that $o(f)=o$. 
\end{theorem}

We observe that the order of growth $[n^2 log(n)]$ belongs to $\CL$ and therefore, the following corollary holds.

\begin{corollary}\label{coroPolEqu}
There exists $f_1,f_2\in \H$ such that $o(f_1)=[n^2]$ and $o(f_2)=[n^2 log(n)]$. In particular, $h_{pol}(f_1)=h_{pol}(f_2)$ yet $o(f_1)\neq o(f_2)$. 
\end{corollary}

The first part of theorem \ref{theoHLR} implies that polynomial orders of growth is the right scope to analyze the dispersion of orbits for maps in $\H$. However, we argue, in the light corollary \ref{coroPolEqu}, that generalized entropy allows to appreciate on a deeper level how rich the universe of possible complexities actually is.

Let us comment on the hypothesis $\CL$. There is a weaker property, which an order of growth may verify, that we call the bounded jump property (BJP) and was introduced in \cite{CoPa23}. We shall use $\B$ to represent the elements in $\OG$ that verifies BJP and $\CB$ to represent the countable completion of $\B$.  The BJP is in fact necessary: for any $f:M\to M$ with $M$ a compact metric space, $o(f)\in \CB$ (see proposition \ref{propB}). Therefore, the best version of theorem \ref{teoFlex} that may be true is for any $o\in \CB$. The linearly invariant property implies the bounded jump property  ($\L\subset \B$) but the converse is false. For example, any exponential order of growth $[e^{tn}]$ fails to verify LIP. The LIP is in fact associated to sub-exponential orders of growth (see proposition \ref{propLIPProp}), but even in this context, we believe the construction of elements in $[a(n)]\in \B\setminus \L$ is possible such that $ [a(n)]< \sup(\Pol)$. On the other hand, any simple example of $[a(n)]< \sup(\Pol)$ constructed by a formula verify LIP. For instance, $[n^t]$, $[log(n)]$, $[log(log(n))]$ belong to $\L$. Moreover, we consider it an nice hypothesis to work with. Since it is naturally associated to sub-exponential orders of growth, any proof involving LIP will use arguments intrinsic to the world of vanishing entropy that will not be an adaptation from the positive entropy setting. 

We would like to finish this introduction with some comments and open questions. Consider $M$ a compact metric space and $f:M\to M$ a homeomorphism. A simple consequence of the variational principle is that $h(f_{|\Omega(f)})=h(f)$. However, for generalized entropy it is only true if $o(f_{|\Omega(f)})\leq o(f)$. In order to have a clean notation, we call $o(f,\Omega(f))= o(f_{|\Omega(f)})$. The inequality must hold because the complexity of a subsystem is always less or equal than the system itself. On the other hand, it was observed in \cite{CoPu21} that the generalized entropy may ``spike" from $o(f,\Omega(f))$ to $o(f)$. Moreover, for the class of homeomorphisms worked in this article $\H$, every map verifies $o(f,\Omega(f))$ is the class of the constant sequence $[c]$ and $[n]\leq o(f) \leq \sup(\Pol)$. Therefore, $o(f,\Omega(f))< o(f)$. We consider that a possible path to understand these jumps in the general case is the following. Define $\hat M = M/\Omega(f)$, that is, collapse the non-wandering set of $f$ to a single point. The map $f$ induces a homeomorphism $\hat f: \hat M \to \hat M$, whose non-wandering set contains only one point. Now,  with the examples we have worked so far, the inequalities 
\[\max\{o(f,\Omega(f)), o(\hat f) \}\leq o(f) \leq o(f,\Omega(f))o(\hat f), \]
always seems to hold and we ask if this is true in general. It is simple to construct examples where one inequality is strict. If the answer is positive and $o(f,\Omega(f))=[c]$, then $o(f)= o(\hat f)$ because $[c]$ is the identity for the product in $\COG$. Therefore, a simpler version of our previous question is the following: If $o(f,\Omega(f))=[c]$, is $o(f)= o(\hat f)$?

This work is structured as follows:
\begin{itemize}
\item In section \ref{secPrelim}, we show the construction of the set of orders growth and generalized entropy. We also explain the coding of orbits for maps in $\H$, and prove some simple results about the bounded jump and linearly invariant properties. 
\item In section \ref{secTeoFlex}, we prove theorem \ref{teoFlex}. We split the proof in two cases. We address first the case when $o\in \B$ and then the general case $o\in \CB$.   
\end{itemize}
\section{Preliminaries}\label{secPrelim}

\subsection{Orders of growth and generalized entropy}\label{subOGGE}

Let us briefly recall how the complete set of the orders of growth and the generalized entropy of a map are defined in \cite{CoPu21}. First, we consider the space of non-decreasing sequences in $[0,\infty)$: $$\mathcal{O}=\{a:\mathbb{N}\rightarrow [0,\infty):a(n)\leq a(n+1),\, \forall n\in \mathbb{N}\}.$$
Next, we define the equivalence relationship $\approx$ in $\OR$ by $a(n)\approx b(n)$ if and only if there exist $c_1,c_2\in (0,\infty)$ such that $c_1 a(n)\leq b(n)\leq c_2 a(n)$ for all $n\in \mathbb{N}$. Since two sequences are related, if both have the same order of growth, we call the quotient space $\displaystyle \mathbb{O}=\mathcal{O}/_{\approx}$ the space of the orders of growth. If $a(n)$ belongs to $\mathcal{O}$, we are going to denote $[a(n)]$ as the associated class in $\mathbb{O}$. If a sequence is defined by a formula (for example, $n^2$), then the order of growth associated will be represented by the formula between the brackets ($[n^2]\in \mathbb{O}$).

We define in $\mathbb{O}$ a natural partial order. We say that $[a(n)] \leq [b(n)]$ if there exists $C>0$ such that $a(n) \leq Cb(n)$, for all $n\in\mathbb{N}$. Observe that $\leq$ is well defined. We consider $\overline{\mathbb{O}}$ the Dedekind-MacNeille completion. This is the smallest complete lattice that contains $\mathbb{O}$. In particular, it is uniquely defined and we will consider $\mathbb{O}\subset \overline{\mathbb{O}}$. We will also call $\overline{\mathbb{O}}$ the complete set of the orders of growth. 

Now, we proceed to define the generalized entropy of a dynamical system in the complete space of the orders of growth. Given $(M,d)$ is a compact metric space and $f:M\rightarrow M$ a continuous map, we define in $M$ the distance 
\[d^{f}_{n}(x,y)=\max \{d(f^k(x),f^k(y)); 0\leq k \leq n-1\},\]
 and we denote the dynamical ball as $B(x,n,\e)=\{y\in M; d^{f}_{n}(x,y)<\e\}$. A set $G\subset M$ is a $(n,\e)$-generator if $\displaystyle M=\cup_{x\in G} B(x,n,\e)$. We define $g_{f,\e}(n)$ as the smallest possible cardinality of a finite $(n,\e)$-generator. If we fix $\e>0$, then $g_{f,\e}(n)$ is a non-decreasing sequence of natural numbers, and thus, $g_{f,\e}(n) \in \mathcal{O}$. For a fixed $n$, if $\e_1<\e_2$, then $g_{f,\e_1} (n) \geq g_{f, \e_2}(n)$, and therefore, $[g_{f,\e_1}(n)]\geq [g_{f,\e_2}(n)]$ in $\mathbb{O}$. We define the generalized entropy of $f$ as 
\[o(f)=\text{\textquotedblleft}\lim_{\e\rightarrow 0}"[g_{f,\e}(n)] =\sup \{[g_{f,\e}(n)]\in \mathbb{O}:\e>0\}\in \overline{\mathbb{O}}. \]

This object is a dynamical invariant.
\begin{theorem}[Correa-Pujals \cite{CoPu21}]\label{TeoCoPu01}
	Let $M$ and $N$ be two compact metric spaces and $f:M\to M$ and $g:N\to N$ be two continuous maps. Suppose there exists $h:M\to N$, a homeomorphism, such that $h\circ f = g \circ h$. Then, $o(f)=o(g)$.
\end{theorem}

We also define the generalized entropy through the point of view of $(n,\e)$-separated. We say that $E\subset M$ is $(n,\e)$-separated if $B(x,n,\e)\cap E = \{x\}$, for all $x\in E$. We define $s_{f,\e}(n)$ as the maximal cardinality of a $(n,\e)$-separated set. Analogously, if we fix $\e>0$, then $s_{f,\e}(n)$ is a non-decreasing sequence of natural numbers. Again, for a fixed $n$, if $\e_1<\e_2$, then $s_{f,\e_1} (n) \geq s_{f, \e_2}(n)$, and therefore, $[s_{f,\e_1}(n)]\geq [s_{f,\e_2}(n)]$. By standard arguments, it is verified that
\[o(f)=\sup \{[s_{f,\e}(n)]\in \mathbb{O}:\e>0\}.\]

The order of growth defined through open coverings is also equivalent. Consider $\U=\{U_1,\cdots, U_k\}$ an open covering of $M$ and define 
\[\U^n=\{U_{i_0}\cap f^{-1}(U_{i_1})\cap \cdots \cap f^{-n}(U_{i_n})\neq \emptyset: i_0,\cdots, i_n\in \{1,\cdots, k\}\},\]
which is also an open covering of $M$. Next, we define the sequence $a_{f,\U}(n)$ as the cardinal of the smallest sub-covering of $\U^n$. Then, it holds
\[o(f)=\sup\{[a_{f,\U}(n)]: \U\text{ is a finite open covering of }M\}.\]

Now, we shall explain how the generalized topological entropy is related to the classical notion of topological entropy and polynomial entropy. Given a dynamical system $f$, we recall that the topological entropy of $f$ is 
\[h(f) = \lim_{\e\to 0}\limsup_{n\rightarrow \infty} \frac{log(g_{f,\e}(n))}{n},\]
and the polynomial entropy of $f$ is
\[h_{pol}(f) = \lim_{\e\to 0}\limsup_{n\rightarrow \infty} \frac{log(g_{f,\e}(n))}{log(n)}.\]

We define the family of exponential orders of growth as the set $\mathbb{E}=\{[\exp(tn)]; t \in(0,\infty) \} \subset \mathbb{O}$ and the family of polynomial orders of growth as the set $\mathbb{P}=\{[n^t];t\in(0,\infty)\}$. Given $o\in \overline{\mathbb{O}}$ we define the intervals $I(o,\mathbb{E})=\{t\in (0,\infty):o\leq [\exp(tn)]\}$ and $I(o,\mathbb{P})=\{t\in (0,\infty):o\leq [n^t]\}$. With these intervals, we define the projections  $\pi_{\mathbb{E}}:\overline{\mathbb{O}}\rightarrow [0,\infty]$ and $\pi_{\mathbb{P}}:\overline{\mathbb{O}}\rightarrow [0,\infty]$ by
\[
\pi_{\mathbb{E}}(o)=
\begin{cases} \inf(I(o,\mathbb{E}))&\text{ if }I(o,\mathbb{E})\neq \emptyset \\
\infty &\text{ if } I(o,\mathbb{E})= \emptyset
\end{cases}
\] 
and 
\[
\pi_{\mathbb{P}}(o)=
\begin{cases} \inf(I(o,\mathbb{P}))&\text{ if }I(o,\mathbb{P})\neq \emptyset \\
\infty &\text{ if } I(o,\mathbb{P})= \emptyset
\end{cases}
\] 

The projection $\pi_{\Exp}$ verifies for elements in $\OG$, $\pi_{\Exp}([a(n)])=\limsup_{n\rightarrow \infty} \frac{log(a(n))}{n}$. Generalized entropy, polynomial entropy and classical entropy are related by the following result. 

\begin{theorem}[Correa-Pujals \cite{CoPu21}]\label{teo122}
	Let $M$ be a compact metric space and $f : M \rightarrow M$, a continuous map. Then, $\pi_\Exp(o(f))=h(f)$ and $\pi_\Pol(o(f))=h_{pol}(f)$.
\end{theorem}

\subsection{Coding of orbits in wandering dynamics}

 Let us consider $M$ a compact metric space and $f:M\to M$ a homeomorphism such that $\Omega(f)=\{p_1,\cdots, p_k\}$. Let $\F$ be a finite family of non-empty subsets of $M\setminus \Omega(f)$. We denote by $\cup \F$ the union of all the elements of $\F$ and by $\infty_{\F}$ the complement of $\cup \F$. Let us fix a positive integer $n$ and consider $\underline{x}=(x_0,\cdots,x_{n-1})$ a finite sequence of points in $M$ and $\underline{w}=(w_0,\cdots,w_{n-1})$ a finite sequence of elements of $\F\cup \{\infty_{\F}\}$. We say that $\underline{w}$ is a coding of $\underline{x}$, relative to $\F$, if for every $i=0,\cdots, n-1$, $x_i\in w_i$. Whenever the family $\F$ is fixed, we simplify the notation by using $\infty$ instead of $\infty_{\F}$. Note that if the sets of $\F$ are not disjoint, we can have more than one coding for a given sequence. We denote the set of all the codings of all orbits $(x,f(x),\cdots, f^{n-1}(x))$ of length $n$ by $\A_{n}(f,\F)$. We define the sequence $c_{f,\F}(n)=\# \A_{n}(f,\F)$, and it is easy to see that $c_{f,\F}(n)\in \mathcal{O}$.

We say that:
\begin{itemize}
\item a set $Y$ is wandering if $f^n(Y)\cap Y =\emptyset$ for every $n\geq 1$. 
\item  $Y$ is a compact neighborhood if it is compact and is the closure of an open set. 
\item the subsets $Y_1,\cdots, Y_L$ of $M\setminus \{\Omega(f)\}$ are mutually singular if, for every $N>0$, there exists a point $x$ and times $n_1,\cdots, n_L$ such that $f^{n_i}(x)\in Y_i$ for every $i=1,\cdots, L$, and $|n_i -n_j|>N$ for every $i\neq j$. 
\end{itemize}

Let us call $\Sigma$ the family of finite families of wandering compact neighborhoods that are mutually singular. Given $\delta>0$, we define $\Sigma_\delta$ as the subset of $\Sigma$ formed by every family whose every element has a diameter smaller than $\delta$. 

\begin{theorem}[Correa - de Paula \cite{CoPa23}]\label{TeoCod}
Let $M$ be a compact metric space and $f:M\rightarrow M$ a homeomorphism such that $\Omega(f)$ is finite. Then, 
\[o(f)=\sup\{[c_{f,\F}(n)]\in \OG: \F\in \Sigma\}.\]
In addition, the equation also holds if we switch $\Sigma$ by $\Sigma_\delta$. 
\end{theorem}

\subsection{Bounded Jump property}\label{subBJP}

We say that a class of orders of growth $[a(n)]$ verifies the bounded jump property (BJP) if there exists a constant $C>0$ such that $a(n+1)\leq Ca(n)$. Note that this definition does not depend on the choice of the class representative. We call $\B\subset \OG$ the set of orders of growth that verify the BJP and $\CB$ the subset in $\COG$ defined by:
\[\CB=\{sup(\Gamma)\in \COG:\Gamma\subset\B\text{ and }\Gamma\text{ is countable}\}.\]
The set $\CB$ is the countable completion of $\B$.

\begin{proposition}\label{propB}
Let us consider $M$ a compact metric space and $f:M\to M$ a homeomorphism. Then, it is verified
\begin{enumerate}
\item For $\U$ a finite covering of $M$,  $[a_{f,\U}(n)]\in \B$.  
\item  $o(f)\in \CB$.
\end{enumerate}  
\end{proposition}

\begin{proof}
We shall begin with proving the result for open coverings. Consider $\U=\{U_1,\cdots, U_k\}$ a finite open covering of $M$ and fix $n\in \N$. Suppose $\V\subset \U^n$ is a sub-covering such that $a_{f,\U}(n)=\# \V$. Then, 
$\V'=\{V\cap f^{-(n+1)}(U_i):V\in \V, U_i\in \U\}\subset \U^{n+1}$ is also a sub-covering with at most $\#\U\#\V$ elements. Therefore, $ a_{f,\U}(n+1)\leq \#\U a_{f,\U}(n)$. 

Now that we know $[a_{f,\U}(n)]\in \B$, we only need to show that $o(f)$ is obtained as the supremum of a countable family of open coverings. By lemma A.1 in \cite{CoPu21}, given $\e>0$ if $diam(\U)\leq \e$, then $[a_{f,\U}(n)]\geq [s_{f,\e}(n)]$. Therefore, if $\{\U_k\}_{k\in\N}$ is a family of open coverings such that $diam(\U_k)\leq \frac{1}{k}$, then $o(f)=\sup\{[a_{f,\U_k}(n)]:k\in\N\}$. 
\end{proof}

In \cite{CoPa23}, we proved an analogous result for $[c_{f,\F}(n)]$.

\begin{lemma}[Correa - de Paula \cite{CoPa23}]\label{LemBouJum}
Let us consider $M$ a compact metric space and $f:M\to M$ a homeomorphism whose non-wandering set is finite. Given $\F\in \Sigma$, the class $[c_{f,\F}(n)]\in \B$.
\end{lemma}

The next result show how we use the BJP. We recall that a set $\S\subset \N$ is syndetic if there exists $N\in \N$ such that for all $n$, the interval $[n,n+N]$ contains at least one point of $\S$. 

\begin{lemma}[Correa - de Paula \cite{CoPa23}]\label{LemSyn}
Let us consider $[a(n)]\in \B$, a syndetic set $\S$ and a sequence $b(n)\in \OR$. If there exist two constants $c_1$ and $c_2$ such that $c_1 b(n)\leq a(n)\leq c_2 b(n)$ for all $n\in \S$, then $[a(n)] =[b(n)]$. 
\end{lemma} 

This lemma tell us that if we know an order of growth in a syndetic set and said order of growth verifies the BJP, then we understand the order of growth in $\N$. For example, if we prove that $c_1 n^k\leq c_{f,\F}(n)\leq c_2 n^k$ for every  $n\in \S$, then $[c_{f,\F}(n)]=[n^k]$.

\subsection{Linearly invariant property} \label{subLIP}

We say that a class of orders of growth $[a(n)]$ verifies the linearly invariant property if there exists an integer $m\geq 2$ such that $[a(n)]= [a(mn)]$. We call $\L\subset \OG$ as the set of orders of growth that verify the LIP and $\CL$ the subset in $\COG$ defined by:
\[\CL=\{sup(\Gamma)\in \COG:\Gamma\subset\L\text{ and }\Gamma\text{ is countable}\}.\]
The set $\CL$ is the countable completion of $\L$. 

Orders of growth in $\L$ verify the following properties

\begin{proposition}\label{propLIPProp}
Let $[a(n)]\in \L$. Then, it is verified:
\begin{enumerate}
\item $[a(n)]\in \B$.
\item For all integer $m\geq 2$, it is verified $[a(mn)]=[a(n)]$.
\item $\pi_{\Exp}([a(n)])=0$.
\end{enumerate} 
\end{proposition}

\begin{proof}
$(1)$. The BJP is equivalent to $[a(n)]=[a(n+1)]$. Suppose $[a(mn)]=[a(n)]$ for some $m\geq 2$. Since $a(n)\leq a(n+1) \leq a(mn)$ for $n\geq 1$, we infer  $[a(n)]=[a(n+1)]$.

$(2)$. Let us suppose $[a(2n)]=[a(n)]$ and consider $c_2>c_1>0$ such that $c_1 \leq \frac{a(2n)}{a(n)} \leq  c_2$. Observe that
\[c_1^2 \leq \frac{a(4n)}{a(2n)} \frac{a(2n)}{a(n)}\leq c_2^2\ \forall n\in \N.\]
From this, we deduce that $[a(4n)]=[a(n)]$ and by induction, we conclude $[a(2^k n)]=[a(n)]$. Now, we fix $m\in \N$ and choose $2^k> m$. Since $a(n)\leq a(mn) \leq a(2^kn)$ and $[a(2^k n)]=[a(n)]$ we see
$[a(mn)]=[a(n)]$.

$(3)$. It is simple to observe that
\begin{align*}
 \pi_{\Exp}([a(n)])  & = \pi_{\Exp}([a(mn)])\\
 & =  \limsup_n \frac{log(a(mn))}{n} \\
  &= m \limsup_n \frac{log(a(mn))}{mn} \\
  &\leq m  \limsup_n \frac{log(a(n))}{n} \\
  & =  m\pi_{\Exp}([a(n)]).
\end{align*}
However, $m\geq 2$ and therefore, $\pi_{\Exp}([a(n)])=0$. 
\end{proof}

A similar reasoning as the one done in proposition \ref{propB} and the third property of the previous proposition implies the following result. 

\begin{proposition}
Consider $f:M\to M$ a dynamical system such that $[a_{f,\U}(n)]\in \L$ for every finite open covering  $\U$. Then,
\begin{enumerate}
\item $o(f)\in \CL$, 
\item and $h(f)=0$. 
\end{enumerate}
\end{proposition}
\section{Proof of theorem \ref{teoFlex}}\label{secTeoFlex}

In this section, we shall prove theorem \ref{teoFlex}. For $o=[n]$, the result holds because for any translation in the plane, the compactification  $f$  verifies $o(f)=[n]$. For this type of maps, all the mutually singular sets consist of a single set and by theorem \ref{TeoCod}, we conclude the existence of a map such that $o(f)=[n]$. 

Now, we are only interested in the case $[n^2]\leq o \leq \sup(\Pol)$. First, we are going to prove the result for the particular case $o\in \L$, and then, adapt to obtain the general result ($o\in\CL$).

\subsection{Construction of the map}

We are going to recall the construction done in \cite{HaRo19} that will provide us a homeomorphism $f:S^2\to S^2$. Let us fix $[a(n)]\in \L$ such that $[n^2]< [a(n)]< \sup(\Pol)$. For said  $[a(n)]$, there must exist $L\geq 3$ such that $[a(n)]\leq [n^L]$. From now on, $L$ is going to be fixed. 

Let us consider $P_1,\cdots, P_{L+2}$ copies of the plane $\R^2$. For each $P_i$ we denote $O_i=\{(x,y)\in P_i: y>0\}$ the upper half plane. For each $i=1,\cdots, L+1$ we are going to select maps $\varphi_i:(0,\infty)\to \R$ and with them, define $\Phi_i:O_i\to O_{i+1}$ by
$\Phi_i(x,y)=(x+\varphi_i(y),y)$.

The first property we want for $\varphi_i$ is the limit in $0$ be $-\infty$. Now, we define $\sim$ the equivalence relation in $\cup P_i$ generated by $(x,y)\sim \Phi_i(x,y)$ if $(x,y)\in O_i$. The quotient space $P=\cup P_i/\sim$  is a Hausdorff simply connected non-compact surface, and thus is homeomorphic to the plane. Figure \ref{imgPlaneP} represents the quotient space $P$.

\begin{figure}[h!]
\begin{center}
\tikzset{every picture/.style={line width=0.75pt}} %set default line width to 0.75pt        

\begin{tikzpicture}[x=0.75pt,y=0.75pt,yscale=-1,xscale=1]

%Curve Lines [id:da8105529734511654] 
\draw    (211.57,59.57) .. controls (245.86,67.29) and (280.71,59.57) .. (281,20.14) ;
\draw [shift={(249.86,60.59)}, rotate = 343.87] [fill={rgb, 255:red, 0; green, 0; blue, 0 }  ][line width=0.08]  [draw opacity=0] (12,-3) -- (0,0) -- (12,3) -- cycle    ;
%Curve Lines [id:da5656492723103168] 
\draw    (289.29,19.86) .. controls (289.57,58.43) and (324.71,68.14) .. (360.43,60.14) ;
\draw [shift={(305.55,54.01)}, rotate = 30.43] [fill={rgb, 255:red, 0; green, 0; blue, 0 }  ][line width=0.08]  [draw opacity=0] (12,-3) -- (0,0) -- (12,3) -- cycle    ;
%Curve Lines [id:da12429852092804805] 
\draw    (359,70.14) .. controls (319.29,74.43) and (313.86,112.43) .. (337,142.14) ;
\draw [shift={(327.05,89.52)}, rotate = 111.71] [fill={rgb, 255:red, 0; green, 0; blue, 0 }  ][line width=0.08]  [draw opacity=0] (12,-3) -- (0,0) -- (12,3) -- cycle    ;
%Curve Lines [id:da13626418538311813] 
\draw    (233.57,142.14) .. controls (253.57,116.71) and (249.86,74.43) .. (210.71,69.86) ;
\draw [shift={(245.42,106.58)}, rotate = 264.24] [fill={rgb, 255:red, 0; green, 0; blue, 0 }  ][line width=0.08]  [draw opacity=0] (12,-3) -- (0,0) -- (12,3) -- cycle    ;
%Curve Lines [id:da7079001280609054] 
\draw    (330.14,149) .. controls (315.29,113.86) and (253,112.43) .. (240.95,150.14) ;
\draw [shift={(293.56,122.96)}, rotate = 185.87] [fill={rgb, 255:red, 0; green, 0; blue, 0 }  ][line width=0.08]  [draw opacity=0] (12,-3) -- (0,0) -- (12,3) -- cycle    ;

% Text Node
\draw (240.29,33.46) node [anchor=north west][inner sep=0.75pt]    {$P_{1}$};
% Text Node
\draw (219.71,90.03) node [anchor=north west][inner sep=0.75pt]    {$P_{2}$};
% Text Node
\draw (276.86,129.54) node [anchor=north west][inner sep=0.75pt]    {$P_{3}$};
% Text Node
\draw (333.43,97.74) node [anchor=north west][inner sep=0.75pt]    {$P_{4}$};
% Text Node
\draw (248.86,8.32) node [anchor=north west][inner sep=0.75pt]    {$-\infty $};
% Text Node
\draw (290.57,8.32) node [anchor=north west][inner sep=0.75pt]    {$+\infty $};
% Text Node
\draw (192.86,70.97) node [anchor=north west][inner sep=0.75pt]    {$-\infty $};
% Text Node
\draw (336.86,124.4) node [anchor=north west][inner sep=0.75pt]    {$-\infty $};
% Text Node
\draw (349.71,40.03) node [anchor=north west][inner sep=0.75pt]    {$-\infty $};
% Text Node
\draw (348.86,74.4) node [anchor=north west][inner sep=0.75pt]    {$+\infty $};
% Text Node
\draw (204.57,123.83) node [anchor=north west][inner sep=0.75pt]    {$+\infty $};
% Text Node
\draw (194.86,40.69) node [anchor=north west][inner sep=0.75pt]    {$+\infty $};
% Text Node
\draw (242.86,144.83) node [anchor=north west][inner sep=0.75pt]    {$-\infty $};
% Text Node
\draw (297.71,144.32) node [anchor=north west][inner sep=0.75pt]    {$+\infty $};
% Text Node
\draw (314.57,34.46) node [anchor=north west][inner sep=0.75pt]    {$P_{5}$};

\end{tikzpicture}
\caption{Gluing of the plane $P$ for $L=3$.} \label{imgPlaneP}
\end{center} 
\end{figure}
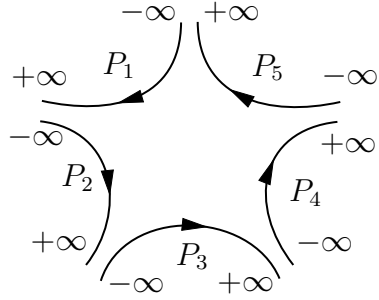

Let $T:\cup P_i\to \cup P_i$ be defined as the translation $T(x,y)=(x+1,y)$ on each $P_i$.  The map $T$ commutes with each $\Phi_i$ and therefore it defines $F:P\to P$ as an orientation preserving homeomorphism on the plane without fixed points. The map $F$ induces a homeomorphism $f$ in the compactification of the plane $\R^2\cup \{\infty\}= S^2$, which verifies $\Omega(f)=\{\infty\}$. In particular, we can apply theorem  \ref{TeoCod} to compute its generalized entropy. Figure \ref{imgMapf} represents the dynamics of $f$.

\begin{figure}[h!]
\begin{center}
\tikzset{every picture/.style={line width=0.75pt}} %set default line width to 0.75pt        

\begin{tikzpicture}[x=0.75pt,y=0.75pt,yscale=-1,xscale=1]

%Curve Lines [id:da30518124245165645] 
\draw    (299.4,141.74) .. controls (244,215.5) and (209.1,128.99) .. (300.34,141.53) ;
%Curve Lines [id:da8950808475599585] 
\draw    (300.34,141.53) .. controls (397.67,142.17) and (352.67,212.5) .. (301.28,141.61) ;
%Shape: Circle [id:dp7083342407214686] 
\draw  [fill={rgb, 255:red, 0; green, 0; blue, 0 }  ,fill opacity=1 ] (297.63,140.42) .. controls (297.63,138.99) and (298.78,137.84) .. (300.21,137.84) .. controls (301.63,137.84) and (302.79,138.99) .. (302.79,140.42) .. controls (302.79,141.84) and (301.63,143) .. (300.21,143) .. controls (298.78,143) and (297.63,141.84) .. (297.63,140.42) -- cycle ;
%Curve Lines [id:da9719941367446978] 
\draw    (298.03,140.42) .. controls (171,131) and (291.07,10.17) .. (300.61,140.42) ;
%Curve Lines [id:da38242804392902285] 
\draw    (300.61,140.42) .. controls (310.2,10.2) and (423.11,131.57) .. (303.19,140.42) ;
%Curve Lines [id:da569057900625008] 
\draw    (300.61,140.42) .. controls (419.97,127.57) and (359.97,243.29) .. (301.82,140.54) ;
%Curve Lines [id:da0472317528556514] 
\draw    (299.4,140.72) .. controls (238.6,245) and (183,122.2) .. (300.61,140.42) ;
%Curve Lines [id:da2970606376255487] 
\draw    (300.21,137.84) .. controls (299.75,109.81) and (294.2,77) .. (273.86,69.57) .. controls (253.51,62.14) and (243.8,70.6) .. (235.57,78.43) .. controls (227.34,86.26) and (228.04,113.98) .. (235.38,123.06) .. controls (242.71,132.14) and (254.88,133.13) .. (254.86,136.07) .. controls (254.84,139.01) and (236.57,140.26) .. (229.57,147.19) .. controls (222.57,154.12) and (215.66,173.26) .. (221.1,182.14) .. controls (226.53,191.03) and (238.57,193.36) .. (246.86,193.36) .. controls (255.14,193.36) and (272.14,178.43) .. (275.29,180.07) .. controls (278.43,181.71) and (269.76,193.98) .. (272.67,207.17) .. controls (275.57,220.36) and (285.51,231.43) .. (302.41,231.24) .. controls (319.32,231.04) and (326.32,221.45) .. (329.43,209.07) .. controls (332.54,196.69) and (324.52,183.02) .. (327.33,180.83) .. controls (330.14,178.64) and (341.46,188.77) .. (349.09,191.74) .. controls (356.71,194.71) and (370.71,192.14) .. (376.57,186.07) .. controls (382.43,180) and (383.86,166.43) .. (378.14,156.64) .. controls (372.43,146.86) and (348.57,140.41) .. (348.57,136.79) .. controls (348.57,133.16) and (358.71,133) .. (366.43,122.71) .. controls (374.14,112.43) and (375,91.29) .. (363.29,78.43) .. controls (351.57,65.57) and (341,63.57) .. (325,71) .. controls (309,78.43) and (300.38,109.69) .. (300.21,137.84) -- cycle ;
%Curve Lines [id:da4196238518994797] 
\draw  [dash pattern={on 0.84pt off 2.51pt}]  (360,101.8) .. controls (436.62,126.5) and (439.3,118.43) .. (439.76,91.59) ;
\draw [shift={(439.8,88.6)}, rotate = 90.78] [fill={rgb, 255:red, 0; green, 0; blue, 0 }  ][line width=0.08]  [draw opacity=0] (8.93,-4.29) -- (0,0) -- (8.93,4.29) -- cycle    ;
%Curve Lines [id:da512673420318001] 
\draw  [dash pattern={on 0.84pt off 2.51pt}]  (311.31,217.34) .. controls (351.12,233.17) and (343.97,236.82) .. (322.4,260.56) ;
\draw [shift={(320.71,262.43)}, rotate = 311.95] [fill={rgb, 255:red, 0; green, 0; blue, 0 }  ][line width=0.08]  [draw opacity=0] (8.93,-4.29) -- (0,0) -- (8.93,4.29) -- cycle    ;
%Curve Lines [id:da8363660989464397] 
\draw  [dash pattern={on 0.84pt off 2.51pt}]  (232.14,171) .. controls (137.89,153.8) and (155.95,190.09) .. (159.02,205.68) ;
\draw [shift={(159.4,208.6)}, rotate = 268.26] [fill={rgb, 255:red, 0; green, 0; blue, 0 }  ][line width=0.08]  [draw opacity=0] (8.93,-4.29) -- (0,0) -- (8.93,4.29) -- cycle    ;
%Curve Lines [id:da1537579559441531] 
\draw  [dash pattern={on 0.84pt off 2.51pt}]  (243.57,93) .. controls (163.77,110.1) and (172.93,93.68) .. (174.05,74.85) ;
\draw [shift={(174.14,71.86)}, rotate = 90] [fill={rgb, 255:red, 0; green, 0; blue, 0 }  ][line width=0.08]  [draw opacity=0] (8.93,-4.29) -- (0,0) -- (8.93,4.29) -- cycle    ;
%Shape: Triangle [id:dp8958619795940812] 
\draw  [fill={rgb, 255:red, 0; green, 0; blue, 0 }  ,fill opacity=1 ] (239.39,74.83) -- (243.32,68.72) -- (246.14,72.16) -- cycle ;
%Shape: Triangle [id:dp7323275141016883] 
\draw  [fill={rgb, 255:red, 0; green, 0; blue, 0 }  ,fill opacity=1 ] (247.99,89.28) -- (251.78,83.09) -- (254.69,86.47) -- cycle ;
%Shape: Triangle [id:dp5439908412868601] 
\draw  [fill={rgb, 255:red, 0; green, 0; blue, 0 }  ,fill opacity=1 ] (260.74,105.52) -- (263.63,98.86) -- (266.98,101.8) -- cycle ;
%Shape: Triangle [id:dp6482573513096812] 
\draw  [fill={rgb, 255:red, 0; green, 0; blue, 0 }  ,fill opacity=1 ] (225.83,187.05) -- (219.23,184.03) -- (222.24,180.74) -- cycle ;
%Shape: Triangle [id:dp3538986571399132] 
\draw  [fill={rgb, 255:red, 0; green, 0; blue, 0 }  ,fill opacity=1 ] (238.31,179.28) -- (231.73,176.21) -- (234.76,172.95) -- cycle ;
%Shape: Triangle [id:dp3410795032956404] 
\draw  [fill={rgb, 255:red, 0; green, 0; blue, 0 }  ,fill opacity=1 ] (251.11,170.61) -- (244.26,168.23) -- (246.94,164.67) -- cycle ;
%Shape: Triangle [id:dp4803111514559295] 
\draw  [fill={rgb, 255:red, 0; green, 0; blue, 0 }  ,fill opacity=1 ] (379.46,181.28) -- (377.17,188.17) -- (373.57,185.53) -- cycle ;
%Shape: Triangle [id:dp3708997709224997] 
\draw  [fill={rgb, 255:red, 0; green, 0; blue, 0 }  ,fill opacity=1 ] (367.75,174.42) -- (364.93,181.11) -- (361.55,178.2) -- cycle ;
%Shape: Triangle [id:dp975132604650822] 
\draw  [fill={rgb, 255:red, 0; green, 0; blue, 0 }  ,fill opacity=1 ] (357.14,167.17) -- (353.87,173.65) -- (350.7,170.52) -- cycle ;
%Shape: Triangle [id:dp41553716907224714] 
\draw  [fill={rgb, 255:red, 0; green, 0; blue, 0 }  ,fill opacity=1 ] (335.21,101.69) -- (342.16,103.78) -- (339.62,107.45) -- cycle ;
%Shape: Triangle [id:dp7180277186575779] 
\draw  [fill={rgb, 255:red, 0; green, 0; blue, 0 }  ,fill opacity=1 ] (341.89,82.35) -- (349.06,83.47) -- (347.06,87.45) -- cycle ;
%Shape: Triangle [id:dp33809986896433] 
\draw  [fill={rgb, 255:red, 0; green, 0; blue, 0 }  ,fill opacity=1 ] (353.9,70.48) -- (360.83,72.65) -- (358.26,76.29) -- cycle ;
%Curve Lines [id:da32436656754150617] 
\draw    (301.82,140.54) .. controls (374.2,245.8) and (229.4,243) .. (300.61,140.42) ;
%Curve Lines [id:da8523403673333396] 
\draw    (298.34,141.53) .. controls (215,124.5) and (284.67,54.83) .. (300.34,141.53) ;
%Curve Lines [id:da3938486210267571] 
\draw    (300.34,141.53) .. controls (313,59.83) and (394.67,124.17) .. (302.34,141.53) ;
%Curve Lines [id:da34419357175849075] 
\draw    (301.28,141.61) .. controls (353.67,220.5) and (251.67,216.17) .. (300.34,141.53) ;
%Curve Lines [id:da15164910719694302] 
\draw  [dash pattern={on 0.84pt off 2.51pt}]  (367.14,175.8) .. controls (416.34,166.99) and (439.05,173.71) .. (439.59,204.63) ;
\draw [shift={(439.57,207.57)}, rotate = 271.47] [fill={rgb, 255:red, 0; green, 0; blue, 0 }  ][line width=0.08]  [draw opacity=0] (8.93,-4.29) -- (0,0) -- (8.93,4.29) -- cycle    ;
%Shape: Triangle [id:dp881720647037963] 
\draw  [fill={rgb, 255:red, 0; green, 0; blue, 0 }  ,fill opacity=1 ] (306.11,218.16) -- (299.2,220.38) -- (299.2,215.92) -- cycle ;
%Shape: Triangle [id:dp30862525793679474] 
\draw  [fill={rgb, 255:red, 0; green, 0; blue, 0 }  ,fill opacity=1 ] (306.47,231.14) -- (299.59,233.46) -- (299.54,229) -- cycle ;
%Shape: Triangle [id:dp6612684939262838] 
\draw  [fill={rgb, 255:red, 0; green, 0; blue, 0 }  ,fill opacity=1 ] (305.28,198.67) -- (298.52,201.3) -- (298.25,196.86) -- cycle ;

% Text Node
\draw (276.69,119.97) node [anchor=north west][inner sep=0.75pt]    {$\infty $};
% Text Node
\draw (109.94,47.09) node [anchor=north west][inner sep=0.75pt]    {$\{( x,0) \in \ P_{1}\}$};
% Text Node
\draw (388.4,60.92) node [anchor=north west][inner sep=0.75pt]    {$\{( x,0) \in \ P_{5}\}$};
% Text Node
\draw (121.2,212.4) node [anchor=north west][inner sep=0.75pt]    {$\{( x,0) \in \ P_{2}\}$};
% Text Node
\draw (263.09,262.8) node [anchor=north west][inner sep=0.75pt]    {$\{( x,0) \in \ P_{3}\}$};
% Text Node
\draw (382.4,209.2) node [anchor=north west][inner sep=0.75pt]    {$\{( x,0) \in \ P_{4}\}$};

\end{tikzpicture}
\caption{ Dynamics of $f$ for $L=3$.} \label{imgMapf}
\end{center} 
\end{figure}
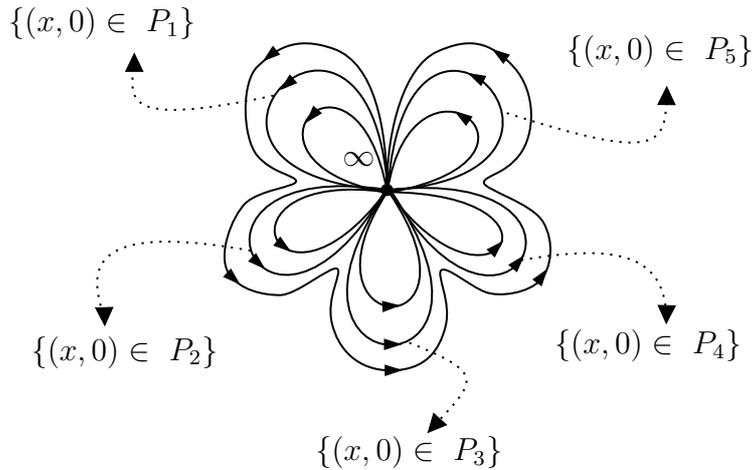

To compute $o(f)$, we need first to understand the location of the mutually singular sets. Let us call $\pi:\cup P_i\to P$ the projection and for each $i$, define $Y_i=\pi([-1/3,1/3]\times [0,1])$ with $[-1/3,1/3]\times [0,1]\subset P_i$. For the map constructed $f$, the family of sets $\{Y_1,\cdots, Y_{L+2}\}$ is mutually singular. Moreover, any other family of mutually singular sets, up to iteration of its elements with a diameter small enough, is equivalent to this one.  Figure \ref{imgMutSingLoc} represents the location of the sets $Y_i$.

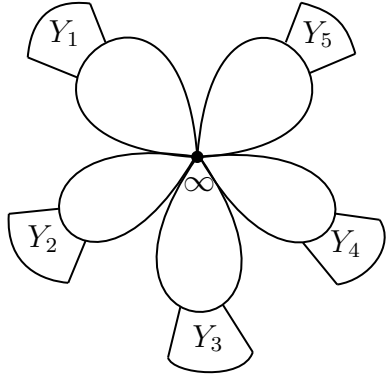
\begin{figure}[h!]
\begin{center}
\tikzset{every picture/.style={line width=0.75pt}} %set default line width to 0.75pt        

\begin{tikzpicture}[x=0.75pt,y=0.75pt,yscale=-1,xscale=1]

%Shape: Circle [id:dp7083342407214686] 
\draw  [fill={rgb, 255:red, 0; green, 0; blue, 0 }  ,fill opacity=1 ] (318.03,160.42) .. controls (318.03,158.99) and (319.18,157.84) .. (320.61,157.84) .. controls (322.03,157.84) and (323.19,158.99) .. (323.19,160.42) .. controls (323.19,161.84) and (322.03,163) .. (320.61,163) .. controls (319.18,163) and (318.03,161.84) .. (318.03,160.42) -- cycle ;
%Straight Lines [id:da19414039517098103] 
\draw    (236,110.36) -- (260.8,120.49) ;
%Straight Lines [id:da29971771451796636] 
\draw    (267.12,83.27) -- (274.76,102.72) ;
%Curve Lines [id:da7435259387728064] 
\draw    (236,110.36) .. controls (237.57,92.21) and (246.14,80.79) .. (267.12,83.27) ;
%Straight Lines [id:da07101586930277581] 
\draw    (373.37,203.64) -- (390.43,224.64) ;
%Straight Lines [id:da7917861878268071] 
\draw    (411.57,192.07) -- (389.04,188.9) ;
%Curve Lines [id:da4133342138016929] 
\draw    (390.43,224.64) .. controls (406.43,222.21) and (419.86,205.07) .. (411.57,192.07) ;
%Straight Lines [id:da5383846402802042] 
\draw    (306,261.36) -- (312.14,235.97) ;
%Straight Lines [id:da8518330838747479] 
\draw    (348.29,258.64) -- (333.43,234.93) ;
%Curve Lines [id:da23566728066124498] 
\draw    (306,261.36) .. controls (308.29,269.79) and (343.29,273.5) .. (348.29,258.64) ;
%Straight Lines [id:da762381646340947] 
\draw    (264.86,202.64) -- (256.22,223.78) ;
%Straight Lines [id:da9039571065100749] 
\draw    (251.57,186.5) -- (226.57,189.07) ;
%Curve Lines [id:da7565934536737868] 
\draw    (256.22,223.78) .. controls (244.44,222.89) and (224.71,216.07) .. (226.57,189.07) ;
%Curve Lines [id:da13749673732521273] 
\draw    (318.03,160.42) .. controls (191,151) and (311.07,30.17) .. (320.61,160.42) ;
%Curve Lines [id:da042927389844338304] 
\draw    (320.61,160.42) .. controls (330.2,30.2) and (443.11,151.57) .. (323.19,160.42) ;
%Curve Lines [id:da05651797713713891] 
\draw    (320.61,160.42) .. controls (439.97,147.57) and (379.97,263.29) .. (321.82,160.54) ;
%Curve Lines [id:da7966979516183663] 
\draw    (319.4,160.72) .. controls (258.6,265) and (203,142.2) .. (320.61,160.42) ;
%Curve Lines [id:da9750553649479012] 
\draw    (321.82,160.54) .. controls (394.2,265.8) and (249.4,263) .. (320.61,160.42) ;
%Straight Lines [id:da2128761647975621] 
\draw    (376.7,118.23) -- (399.46,108.54) ;
%Straight Lines [id:da6681942399159944] 
\draw    (372.38,82.85) -- (364.75,102.78) ;
%Curve Lines [id:da24887927005881938] 
\draw    (399.46,108.54) .. controls (395.46,95.15) and (391,87.62) .. (372.38,82.85) ;

% Text Node
\draw (312.02,168.97) node [anchor=north west][inner sep=0.75pt]    {$\infty $};
% Text Node
\draw (245.1,90) node [anchor=north west][inner sep=0.75pt]    {$Y_{1}$};
% Text Node
\draw (385.39,195.92) node [anchor=north west][inner sep=0.75pt]    {$Y_{4}$};
% Text Node
\draw (316.7,243.4) node [anchor=north west][inner sep=0.75pt]    {$Y_{3}$};
% Text Node
\draw (234.42,193.2) node [anchor=north west][inner sep=0.75pt]    {$Y_{2}$};
% Text Node
\draw (371.63,90.49) node [anchor=north west][inner sep=0.75pt]    {$Y_{5}$};

\end{tikzpicture}
\caption{ Locations of sets $Y_i$ for $L=3$.} \label{imgMutSingLoc}
\end{center} 
\end{figure}

 The key observation to be made is that for a point $(x,y)\in P_i$ such that its projection $\pi(x,y)\in Y_i$ needs for $-\varphi_i(y)>0$ iterates to reach $Y_{i+1}$. This means, that with the height $y$, we can control the amount of time the point $\pi(x,y)\in Y_1$ is going to need to reach each one of the sets $Y_i$.

\subsection{Proof of the theorem for the particular case} \label{secPartCase}

Up to this point, we have followed the construction done \cite{HaRo19}. Our next step will be to modify the choice of the maps $\f_i$  done in \cite{HaRo19} such that $[c_{f,\F}(n)]=[a(n)]$.

Let us construct $\varphi_i$ by steps. We shall chose $\f_1$ such that it is negative, increasing, and, for each positive integer $k_1$, assumes the value $-k_1$ on a non trivial interval $I_{k_1}$. This collection of intervals tends to $0$ when $k_1$ tends to $+\infty$. For convenience, we assume that all these intervals are included in the interval $\left. (0, 1 \right]$. The number $k_1$ is going to be our main variable, and we would like to stress that any point of the form $\pi(x,y)\in Y_1$ with $y\in I_{k_1}$ is going to need $k_1$ iterates of $f$ to reach $Y_2$.

Now we look at $k_1$ iterations of the points in the previous step that are in $Y_2$. These points do have the same height $y$ when looked in $P_2$.  The restriction of $\f_2$ to $I_{k_1}$ tell us how many iterates are needed for these points in $Y_2$ to reach $Y_3$. We are going to ask that $\f_2$ restricted to $I_{k_1}$ is continuous, increasing and varies from $-(k_1+a_2(k_1))$ to $-k_1$ (we postpone for later the choice of the integer $0< a_2(k_1)\leq k_1$). Moreover, we need that for each integer $-k_2$ in the interval  $[-(k_1+ a_2(k_1)),-k_1]$, $\f_2$ assumes the value $-k_2$ on a non trivial sub-interval $I_{k_1,k_2}\subset I_{k_1}$.  Now, any point with height in  $I_{k_1,k_2}$ is going to need $k_1$ iterations to go from $Y_1$ to $Y_2$ and $k_2$ iterations to go from $Y_2$ to $Y_3$. We would like to observe that for fixed $k_1$, there are $a_2(k_1)$ possible iterations needed to reach from $Y_2$ to $Y_3$. Finally, we would like remark that we chose the interval $[-(k_1+ a_2(k_1)),-k_1]$ instead of $[-a_2(k_1),0]$ because we also need that the limit of $\f_2$ in $0$ be $-\infty$.  

Likewise, we define $\f_3$ to be increasing on each interval $I_{k_1,k_2}$ and assume each integer $-k_3$ between $-(3k_1-k_2+a_3(k_1))$ and $-(3k_1-k_2)$ on a non trivial sub-interval $I_{k_1,k_2,k_3}$ of $I_{k_1,k_2}$. To explain why we choose the interval like this, first observe that $2k_1 \leq k_1+ k_2 \leq 3k_1$. By the choice of the interval, if $0< a_3(k_1)\leq k_1$, then,
\[4k_1 \leq k_1+k_2+k_3 \leq 5k_1.\]
This gives us the following control: A point in $Y_1$ with height in $I_{k_1}$ is going to reach $Y_2$ in $k_1$ iterates, reach $Y_3$ in between $2k_1$ to $3k_1$ iterates and $Y_4$ in between $4k_1$ to $5k_1$ iterates. Similar to the previous step, for $k_1$ fixed, there are $a_3(k_1)$ possible iterations needed to reach from $Y_3$ to $Y_4$. 
	
We define $\f_i$ inductively until $\f_{L+1}$. The map $\f_{i+1}$ is chosen such that for $k_1, \cdots, k_i$ fixed, it assumes each integer value $-k_{i+1}$ in a sub-interval $I_{k_1,\cdots,k_{i+1}}$ of $I_{k_1,\cdots,k_{i}}$. The value of $-k_{i+1}$ ranges in an interval of length $a_{i+1}(k_1)\leq k_1$ and such that 
\[2ik_1\leq k_1+\cdots + k_{i+1}\leq (2i+1)k_1.\]
For $i+1=L+1$, we see
\[2Lk_1 \leq  k_1+\cdots + k_{L+1} \leq (2L+1)k_1.\]

Let us finally address the choice of $a_i(n)$. We choose them such that
\[\left[n \sum_{k_1=1}^{n} a_2(k_1)\cdots a_{L+1}(k_1) \right]  =[a(n)].\]

\begin{lemma}
Given $[a(n)]\in \B$ such that $[n^2]< [a(n)] \leq [n^L]$ there exist sequences $a_2(n),$ $\cdots,$ $a_{L+1}(n)$ such that 
\[\left[n \sum_{k=1}^{n} a_2(k)\cdots a_{L+1}(k) \right]  =[a(n)],\]
and $a_i(n)\leq n$ for all $i=2,\cdots, L+1$ and all $n\geq 1$. 
\end{lemma}

\begin{proof}
Let us call $d(n)=  n \sum_{k=1}^{n} a_2(k)\cdots a_{L+1}(k)$ and $e(n)=a_2(n)\cdots a_{L+1}(n)$. Consider two constants $c_1\geq 2$ and $c_2>0$ such that $\frac{a(n+1)}{a(n)}\leq c_1$ and $a(n)\leq c_2 n^L$, for all $n\geq 0$. 

We suppose that we have picked our sequences up to $n$ verifying  $b_1 a(n)\leq d(n) \leq b_2 a(n)$, for some $b_1<b_2$.  
For the next step, we choose $e(n+1)$ such that $b_1 a(n+1)\leq d(n+1) \leq b_2 a(n+1)$. 

If $d(n+1)$ verifies $\frac{d(n+1)}{d(n)}= \frac{a(n+1)}{a(n)}$, then 
\[b_1 a(n+1) = b_1a(n) \frac{a(n+1)}{a(n)}\leq  d(n) \frac{d(n+1)}{d(n)} \leq  b_2a(n) \frac{a(n+1)}{a(n)}  = b_2 a(n+1).\]
This implies the desired property for $d(n+1)$. We need to see that this pick is compatible with our restriction $a_i(n)\leq n$ for all $i$ and all $n$. Once we observe that
\[\frac{d(n+1)}{d(n)} = \frac{n+1}{n} + e(n+1) \frac{n+1}{d(n)},\]
we infer that we need to define $e(n+1)$ by
\[e(n+1)= \left( \frac{a(n+1)}{a(n)} - \frac{n+1}{n}\right) \frac{d(n)}{n+1}.\]
Yet, 
\[ \left( \frac{a(n+1)}{a(n)} - \frac{n+1}{n}\right) \frac{d(n)}{n+1}\leq (C-2)b_2 c_2 \frac{n^L}{n+1}\leq (n+1)^{L},\]
for $n$ large enough. Since we have $L$ sequences $a_i$, we can pick $e(n+1)$ satisfying all the above.  

\end{proof}

Our task now is to compute $[c_{f,\F}(n)]$. 

Consider a word $w \in \A_{n}(f,\F)$ and $x\in S^2$ a point such that $w$ is the coding of the itinerary of $x$. We call the displacement of $w$, $k$ units to the right, to the word of length $n$ that codes the itinerary of $f^{-k}(x)$. The displacement of $w$, $k$ units to the left, is the word of length $n$ that codes the itinerary of $f^{k}(x)$.

Consider the syndetic set $\S=(2L+2)\N$. In light of lemma \ref{LemSyn}, if we prove that  $d_1 a(n) \leq c_{f,\F}(n) \leq d_2 a(n)$ for $n\in S$, then $[c_{f,\F}(n)]=[a(n)]$.

Let us consider $ b_2>b_1 >0$ such that 
\[b_1 a(n) \leq n \sum_{k_1=1}^{n} a_2(k_1)\cdots a_{L-1}(k_1)  \leq b_2 a(n)\ \forall n\in \N.\]
By linear invariance of $[a(n)]$ there exists $c_2> c_1>0$ such that
\[ c_1 a(k(2L+2)) \leq a(k) \leq c_2 a(k(2L+2))\ \forall k\in \N.\]

We now fix some $n\in \S$ and define $k\in \N$ such that $n=k(2L+2)$. If a word in $\A_{n}(f,\F)$ begins with $Y_1$ and is associated to $k_1\leq k$, then all the symbols $Y_2,\cdots, Y_{L+2}$ appear in $w$. By our construction, there are  $a_2(k_1)\cdots a_{L+1}(k_1)$ of those. Each one of this words can be displaced at least $k$ to the right creating new unique words in $\A_{n}(f,\F)$. This is true because all symbols $Y_i$ appear before $(2L+1)k_1 \leq (2L+1)k$. From this, we infer that in $\A_n(f,\F)$ there are at least  $k \sum_{k_1=1}^k  a_2(k_1)\cdots a_{L+1}(k_1)$ distinct words. Then,
\[c_{f,\F}(n)\geq k \sum_{k_1=1}^k  a_2(k_1)\cdots a_{L+1}(k_1)\geq b_1 a(k)\geq b_1 c_1 a(n).\]

On the other hand, every word in $\A_n(f,\F)$ is obtained from displacing a word that starts with $Y_1$. We can displace up to $n$  spaces  to the right or $(2L+2)n$ to the left. We only need to count up to $k_1 =n$ and therefore,    
\[c_{f,\F}(n)\leq (2L+3)n \sum_{k_1=1}^n  a_2(k_1)\cdots a_{L+1}(k_1)\leq (2L+3) c_2 a(n).\]
In conclusion, if $d_1= b_1c_1$ and $d_2=(2L+3)c_2$, we have proved that
\[d_1 a(n) \leq c_{f,\F}(n) \leq d_2 a(n)\ \forall n\in \S.\]
Thus, 
\[ [c_{f,\F}(n)]=[a(n)]\]
and with this, we have finished the proof of theorem \ref{teoFlex} for the case $o=[a(n)]\in \L$.

\subsection{Lemmas of orders of growth}

To prove the general case of theorem \ref{teoFlex}, the following lemma will simplify our reasoning. 
\begin{lemma}\label{lemSupOrd}
	Let $\Gamma\subset \mathbb{O}$ a countable subset of orders of growth, then there exists a countable and ordered subset $\hat \Gamma\subset \mathbb{O}$ such that $\sup (A) = \sup (B)$. Moreover, if $\Gamma\subset \B$ or $\Gamma\subset \L$, then so does $\hat\Gamma$. 
\end{lemma}

In order to prove lemma \ref{lemSupOrd}, the following is necessary. 

\begin{lemma}\label{lemSup}
Suppose that  $[a(n)],[b(n)] \in \OG$. Then:
\begin{enumerate} 
\item $\sup\{[a(n)],[b(n)]\}=[\max\{a(n),b(n)\}]$.
\item If  $[a(n)],[b(n)] \in \B$ or  $[a(n)],[b(n)] \in \L$, so does $\sup\{[a(n)],[b(n)]\}$.
\end{enumerate}
\end{lemma}

\begin{proof}
	(1) We know that $[a(n)]\leq [\max\{a(n),b(n)\}] $ and $[b(n)]\leq [\max\{a(n),b(n)\}]$, then $\sup\{[a(n)],[b(n)]\}\leq [\max\{a(n),b(n)\}]$. But, if we suppose $$\sup\{[a(n)],[b(n)]\}< [\max\{a(n),b(n)\}],$$ then there exists $[c(n)]\in \OG$ such that $$\sup\{[a(n)],[b(n)]\}<[c(n)]< [\max\{a(n),b(n)\}].$$ 
By definition, we know  $[a(n)]\leq\sup\{[a(n)],[b(n)]\},$ and then $[a(n)]\leq [c(n)]$. Analogously, $[b(n)]\leq [c(n)]$. Thus, there must exist constants $d_1,d_2>0$ such that $a(n)\leq d_1 c(n)$ and $b(n)\leq d_2 c(n)$. Let $d=\max\{d_1,d_2\}$ and observe $\max\{a(n),b(n)\}\leq dc(n)$, which implies $[\max\{a(n),b(n)\}]\leq [c(n)]$, a contradiction. Therefore, $$\sup\{[a(n)],[b(n)]\}=[\max\{a(n),b(n)\}],$$ as we wanted.

$(2)$ For elements in $\B$ is an immediate conclusion from part (1) of this lemma. For elements in $\L$, it is necessary to use point $(2)$ of proposition \ref{propLIPProp}. 
\end{proof}

\begin{proof}[Proof of lemma \ref{lemSupOrd}]
	Let $\Gamma=\{[b_1(n)],[b_2(n)],\cdots,[b_k(n)],\cdots\}\subset \mathbb{O}$, we construct the subset $\hat \Gamma$ as follows: 
$$\begin{array}{ccl}
[a_1(n)]&=&[b_1(n)], \\ 

[a_2(n)]&=&\sup\{[b_1(n)],[b_2(n)]\}=[\max\{b_1(n),b_2(n)\}], \\ 

[a_3(n)]&=&\sup\{[a_1(n)],[b_3(n)]\}=\sup\{\sup\{[b_1(n)],[b_2(n)]\},[b_3(n)]\}\\ 

 &=&\sup\{[b_1(n)],[b_2(n)],[b_3(n)]\} =[\max\{b_1(n),b_2(n),b_3(n)\}], \\
 
\vdots & & \\

[a_k(n)]&=&\sup\{[a_{k-1}(n)],[b_{k}(n)]\}=\cdots = [\max\{b_1(n),b_2(n),\cdots, b_{k}(n)\}], \\

\vdots
\end{array}$$
The previous identities hold by lemma \ref{lemSup} and an inductive argument. $A$ is clearly a countable set, and it is easy to see that it is an ordered set, in fact
$$\begin{array}{rcl} [a_1(n)]&=&\sup\{[b_1(n)],[b_2(n)]\}\leq\sup\{[b_1(n)],[b_2(n)],[b_3(n)]\}\\
&=&[a_2(n)]\leq \cdots\leq  [a_k(n)] \leq [a_{k+1}(n)] \leq \cdots .
\end{array}$$
If $\Gamma\subset \B$ or $\Gamma\subset \L$, by the second item of $\ref{lemSup}$, so does $\hat \Gamma$. 	
	
We shall prove now that $\sup (\hat \Gamma)= \sup (\Gamma)$. Since $\sup (\hat \Gamma) \geq [a_k(n)]\geq [b_k(n)]$, we deduce $\sup (\hat \Gamma) \geq \sup (\hat \Gamma)$. On the other hand, $\sup (\Gamma) \geq [b_k(n)]$, for all $k\in \N$, and by definition of $[a_k(n)]$, we infer $\sup (\Gamma)\geq [a_k(n)]$ for all $k\in \N$. Therefore, $\sup (\Gamma)\leq \sup (\hat \Gamma)$.
\end{proof}

\subsection{Proof of the general case}

Let us consider $\Gamma \subset \L$ countable such that $[n^2]\leq o=\sup(\Gamma)\leq \sup(\Pol)$. We want to construct $f$ such that $o(f)=o$. 

By lemma \ref{lemSupOrd}, we might suppose that $\Gamma$ is ordered. Moreover, if $\Gamma=\{[a_k(n)]:k\in \N\}$, we may also assume that $[a_1(n)]\geq [n^2]$.

We first construct $f_1\in \H$ such that $o(f_1)=[a_1(n)]$ using the technique of subsection \ref{secPartCase}. We shall call $\infty$ to the only fixed point of $f_1$ and $D_1$ as the disk associated to the lower semi-plane in $P_1$. We also call $S_1$ to the closure of its complement. The dynamics constructed in subsection \ref{secPartCase} lay entirely in the region $S_1$. Therefore, $o(f_{1|S_1})=[a_1(n)]$.

\begin{figure}[!h]\label{disf2}
\begin{center}
\tikzset{every picture/.style={line width=0.75pt}} %set default line width to 0.75pt        

\begin{tikzpicture}[x=0.75pt,y=0.75pt,yscale=-1,xscale=1]

%Curve Lines [id:da8105529734511654] 
\draw    (108,57.17) .. controls (124.33,79.5) and (162.07,83.64) .. (162.35,44) ;
\draw [shift={(135.91,73.82)}, rotate = 359.3] [fill={rgb, 255:red, 0; green, 0; blue, 0 }  ][line width=0.08]  [draw opacity=0] (10.72,-5.15) -- (0,0) -- (10.72,5.15) -- (7.12,0) -- cycle    ;
%Curve Lines [id:da5656492723103168] 
\draw    (226.67,78.83) .. controls (187.67,91.5) and (188.33,120.17) .. (233,137.5) ;
\draw [shift={(198.51,103.16)}, rotate = 86.49] [fill={rgb, 255:red, 0; green, 0; blue, 0 }  ][line width=0.08]  [draw opacity=0] (10.72,-5.15) -- (0,0) -- (10.72,5.15) -- (7.12,0) -- cycle    ;
%Curve Lines [id:da13626418538311813] 
\draw    (233.33,157.83) .. controls (164.33,109.17) and (135,109.17) .. (73,157.83) ;
\draw [shift={(159.13,121.86)}, rotate = 184.37] [fill={rgb, 255:red, 0; green, 0; blue, 0 }  ][line width=0.08]  [draw opacity=0] (10.72,-5.15) -- (0,0) -- (10.72,5.15) -- (7.12,0) -- cycle    ;
%Curve Lines [id:da4850842944262794] 
\draw    (126.67,215.83) .. controls (139.67,194.17) and (166,193.5) .. (165.33,222.5) ;
\draw [shift={(144.27,200.92)}, rotate = 352.09] [fill={rgb, 255:red, 0; green, 0; blue, 0 }  ][line width=0.08]  [draw opacity=0] (10.72,-5.15) -- (0,0) -- (10.72,5.15) -- (7.12,0) -- cycle    ;
%Curve Lines [id:da6830177286463675] 
\draw    (206.67,196.5) .. controls (198.33,181.5) and (204.33,150.83) .. (230.33,163.5) ;
\draw [shift={(203.85,175.03)}, rotate = 288.5] [fill={rgb, 255:red, 0; green, 0; blue, 0 }  ][line width=0.08]  [draw opacity=0] (10.72,-5.15) -- (0,0) -- (10.72,5.15) -- (7.12,0) -- cycle    ;
%Curve Lines [id:da8979306236335514] 
\draw    (171,222.17) .. controls (165.67,198.83) and (187.33,179.83) .. (202,203.17) ;
\draw [shift={(174.01,200.66)}, rotate = 310] [fill={rgb, 255:red, 0; green, 0; blue, 0 }  ][line width=0.08]  [draw opacity=0] (10.72,-5.15) -- (0,0) -- (10.72,5.15) -- (7.12,0) -- cycle    ;
%Shape: Circle [id:dp28811964079698127] 
\draw  [fill={rgb, 255:red, 0; green, 0; blue, 0 }  ,fill opacity=1 ] (379.9,122.3) .. controls (379.9,121.25) and (380.75,120.4) .. (381.8,120.4) .. controls (382.85,120.4) and (383.7,121.25) .. (383.7,122.3) .. controls (383.7,123.35) and (382.85,124.2) .. (381.8,124.2) .. controls (380.75,124.2) and (379.9,123.35) .. (379.9,122.3) -- cycle ;
%Curve Lines [id:da9630686419392767] 
\draw    (382.79,122.56) .. controls (490.56,225) and (281,222.33) .. (380.79,122.19) ;
%Curve Lines [id:da6952983011098728] 
\draw    (381.8,122.3) .. controls (520.87,100.6) and (380.47,25) .. (381.8,120.4) ;
%Curve Lines [id:da414714992696418] 
\draw    (381.8,120.4) .. controls (380.07,15.8) and (249.27,100.2) .. (381.8,122.3) ;
%Curve Lines [id:da63508389820832] 
\draw    (380.29,123.44) .. controls (364.42,167.06) and (349.66,159.26) .. (378.79,123.81) ;
%Curve Lines [id:da22495593911989342] 
\draw    (381.8,124.2) .. controls (389.29,162.69) and (402.79,164.81) .. (382.79,123.31) ;
%Curve Lines [id:da16912708964122225] 
\draw    (382.79,123.31) .. controls (397.36,157.62) and (414.9,160.38) .. (382.79,122.56) ;
%Curve Lines [id:da7915056796916238] 
\draw    (175,46.5) .. controls (168,67.17) and (183.33,88.5) .. (221,70.5) ;
\draw [shift={(181.25,73)}, rotate = 27.81] [fill={rgb, 255:red, 0; green, 0; blue, 0 }  ][line width=0.08]  [draw opacity=0] (10.72,-5.15) -- (0,0) -- (10.72,5.15) -- (7.12,0) -- cycle    ;
%Curve Lines [id:da06174169363528992] 
\draw    (102,67.73) .. controls (127.33,100.83) and (114.33,114.17) .. (75,140.5) ;
\draw [shift={(107.51,114.63)}, rotate = 306.28] [fill={rgb, 255:red, 0; green, 0; blue, 0 }  ][line width=0.08]  [draw opacity=0] (10.72,-5.15) -- (0,0) -- (10.72,5.15) -- (7.12,0) -- cycle    ;
%Curve Lines [id:da6193229855841036] 
\draw    (89,193.83) .. controls (105,179.5) and (125.67,189.5) .. (119,212.5) ;
\draw [shift={(106.83,187.31)}, rotate = 18.1] [fill={rgb, 255:red, 0; green, 0; blue, 0 }  ][line width=0.08]  [draw opacity=0] (10.72,-5.15) -- (0,0) -- (10.72,5.15) -- (7.12,0) -- cycle    ;
%Curve Lines [id:da06319156006210047] 
\draw    (77.67,162.5) .. controls (93,156.17) and (109.33,173.17) .. (87.67,187.83) ;
\draw [shift={(93.44,163.98)}, rotate = 50.62] [fill={rgb, 255:red, 0; green, 0; blue, 0 }  ][line width=0.08]  [draw opacity=0] (10.72,-5.15) -- (0,0) -- (10.72,5.15) -- (7.12,0) -- cycle    ;
%Curve Lines [id:da600082967290146] 
\draw    (383.7,122.3) .. controls (455.25,217.04) and (504.87,89) .. (381.8,122.3) ;
%Curve Lines [id:da0010692790397923702] 
\draw    (379.29,122.94) .. controls (315.27,210.2) and (268.87,89) .. (381.8,122.3) ;
%Curve Lines [id:da9686102352343344] 
\draw    (381.17,124.06) .. controls (373.92,158.69) and (390.67,172.31) .. (381.8,124.2) ;
%Curve Lines [id:da07465915682499946] 
\draw    (380.29,123.44) .. controls (363.04,164.31) and (375.17,171.19) .. (381.42,122.94) ;

% Text Node
\draw (146.29,130.84) node [anchor=north west][inner sep=0.75pt]    {$P_{1}$};
% Text Node
\draw (210.95,94.51) node [anchor=north west][inner sep=0.75pt]    {$P_{2}$};
% Text Node
\draw (192.24,42.92) node [anchor=north west][inner sep=0.75pt]    {$P_{3}$};
% Text Node
\draw (214.62,171.84) node [anchor=north west][inner sep=0.75pt]    {$P_{1}^{2}$};
% Text Node
\draw (180.29,201.51) node [anchor=north west][inner sep=0.75pt]    {$P_{2}^{2}$};
% Text Node
\draw (137.62,210.51) node [anchor=north west][inner sep=0.75pt]    {$P_{3}^{2}$};
% Text Node
\draw (352.81,118.66) node [anchor=north west][inner sep=0.75pt]    {$\infty $};
% Text Node
\draw (374.46,161.54) node [anchor=north west][inner sep=0.75pt]    {$D_{1}$};
% Text Node
\draw (427.23,169.33) node [anchor=north west][inner sep=0.75pt]    {$S_{1}$};
% Text Node
\draw (122.57,37.59) node [anchor=north west][inner sep=0.75pt]    {$P_{4}$};
% Text Node
\draw (61.9,90.92) node [anchor=north west][inner sep=0.75pt]    {$P_{5}$};
% Text Node
\draw (93.95,195.51) node [anchor=north west][inner sep=0.75pt]    {$P_{4}^{2}$};
% Text Node
\draw (68.62,164.84) node [anchor=north west][inner sep=0.75pt]    {$P_{5}^{2}$};

\end{tikzpicture}
\caption{$f_2:S^2 \rightarrow S^2$.}
\end{center}
\end{figure}
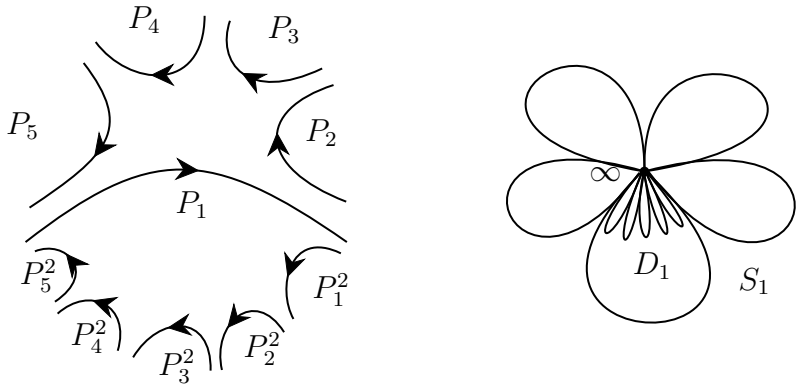

We now define $f_2\in \H$ such that $f_{2|S_1}=f_{1|S_1}$ and $f_{2|D_1}$ verifies $o( f_{2|D_1})=[a_2(n)]$. The construction is made in an analogous way by adding new planes $P^2_i$ and using the lower semi-plane of $P_1$. Figure \ref{disf2} represents the construction of $f_2$.

We choose one of the new planes and call $D_2$ to the region associated to its lower half-plane. We call $S_2$ to the closure of the complement of $D_2$ and observe that dynamics created so far lay entirely in $S_2$. In particular, 
\[o(f_{2|S_2})=o(f_2)=\sup\{o(f_{2|S_1}), o(f_{2|D_1})\}=\sup \{[a_1(n)],[a_2(n)]\}= [a_2(n)].\]

Inductively, we define a family of regions  $D_k$ and $S_k=\overline{D_k^c}$ where $D_k$ is a disk, contains $\infty$ in its border, $D_{k+1}\subset D_k$, and we may also ask that $diam(D_k)\to 0$. By induction, we also have a family of maps $f_k\in \H$ such that:
\begin{itemize}
\item  $f_k(S_i)=S_i$ for all $i\leq k$. 
\item  $f_{k+1|S_k}= f_{k|S_k}$.
\item  $o(f_k)= o(f_{k|S_k})=[a_k(n)]$. 
\end{itemize}

This family of dynamical systems in $\H$ define a unique map $f\in \H$ with $f(x)=f_k(x)$ if $x\in S_k$.

Let us check if $o(f)=\sup(\Gamma)$. We first observe that
\[o(f)\geq o(f_{|S_k})= o(f_{k|S_k})=[a_k(n)], \]
thus, $o(f)\geq\sup (\Gamma)$. To see the other inequality, we chose $\e>0$ and $k$ such that $diam(D_k)\leq \e/2$. Since $f(D_k)= D_k$, for any $x_k\in D_k$, the dynamical ball $B(x_k,n,\e)$ contains $D_k$ for every $n$. If $G_n$ is an $(n,\e)$ generator set of $f_{|S_k}$ with 
$g_{f_{|S_k},\e}(n) =   \#G_n$, then $G_n\cup \{x_k\}$ is an $(n,\e)$-generator set of $f$. Therefore, 
$g_{f,\e}(n)\leq g_{f_{|S_k},\e}(n) +1 $ and then 
\[ [g_{f,\e}(n)]\leq [g_{f_{|S_k},\e}(n)]\leq o(f_{|S_k} ) = [a_k(n)]\leq \sup (\Gamma).\]
From this, we conclude  that $o(f)\leq \sup (\Gamma)$. This proves $o(f)=\sup (\Gamma)$ and finishes the proof of theorem \ref{teoFlex}.

\bibliographystyle{siam}
\bibliography{Bibliog}

\end{document}